\documentclass{amsart}

\usepackage{amssymb}

\newcommand{\Ps}{\mathbf{P}}
\newcommand{\As}{\mathbf{A}}
\newcommand{\C}{\mathbf{C}}
\newcommand{\Q}{\mathbf{Q}}
\newcommand{\Z}{\mathbf{Z}}

\newcommand{\R}{\mathbf{R}}

\newcommand{\cO}{\mathcal{O}}

\newcommand{\cI}{\mathcal{I}}

\newtheorem{lemma}{Lemma}[section]
\newtheorem{proposition}[lemma]{Proposition}
\newtheorem{theorem}[lemma]{Theorem}
\newtheorem{corollary}[lemma]{Corollary}

\theoremstyle{definition}
\newtheorem{notation}[lemma]{Notation}
\newtheorem{definition}[lemma]{Definition}
\newtheorem{construction}[lemma]{Construction}

\theoremstyle{remark}
\newtheorem{remark}[lemma]{Remark}
\newtheorem{example}[lemma]{Example}
\DeclareMathOperator{\degsp}{degsp}

\DeclareMathOperator{\length}{length}
\DeclareMathOperator{\Gr}{Gr}

\DeclareMathOperator{\sing}{sing}
\DeclareMathOperator{\spe}{sp}

\title{Nodal deformations of hypersurfaces with an ordinary $m$-fold point}
\author{Remke Kloosterman}
\begin{document}
\begin{abstract}
Let $X\subset \Ps^n$ be a nodal hypersurface of degree $d$ with an ordinary $m$-fold point. Let $\delta_X$ be the largest value of $\delta$ such that $X$ is contained in the closure of the Severi variety of degree $d$ hypersurface in  $\Ps^n$ with $\delta$ nodes.

After recalling how the semicontinuity of the spectrum yields an upper  bound for $\delta_X$,  we produce various lower bounds for $\delta_X$.
These bounds are given by polynomials in $m$ of degree $n$,  with different leading coefficients.
\end{abstract}

\thanks{The author would like to thank Orsola Tommasi %and the anonymous referees 
for several comments on a previous version of this paper.
The author is a member  of INdAM-GNSAGA. This work was supported  by the UNIPD BIRD-SID-2024 
project ``Moduli spaces and positivity".}

\maketitle
\section{Introduction}

In this paper we work over the field of complex numbers.
Fix integers $n,d,\delta$ and denote by $V^n_{d,\delta}\subset \Ps(\C[x_0,\dots,x_n]_d)$ the Severi variety of degree $d$ hypersurfaces in  $\Ps^n$ with $\delta$ nodes. For a singular hypersurface $X\subset \Ps^n$  degree $d$ in $\Ps^n$ let $\delta_X$ be the maximal value of $\delta$ such that (the point corresponding to) $X$ lies in the closure of $V^n_{d,\delta}$.

In the case of irreducible plane curves it is known that $\delta_X$ equals the difference between the arithmetic genus of $X$ and geometric genus of $X$.

In this note, we recall an upper bound for $\delta_X$ in the case that $X$ has only isolated singularities and produce various lower bound for $\delta_X$ in case $X$ has a single singularity which is an ordinary $m$-fold point. 
This improves and extends bounds given by  Ciliberto and Galati  in a recent  paper \cite{CiGa26}, in which they discussed approaches for determining $\delta_X$ in the case of surfaces ($n=3$).

In Section~\ref{secUpper} we discuss an upper bound for $\delta_X$ derived from the so-called spectral bound. 
It is mostly an expository section and is intended to translate results from singularity theory into a more algebro-geometric language.
The upper bound for $\delta_X$ is  a consequence of the semi-continuity of the spectrum due to Varchenko \cite{VarSC} and Steenbrink \cite{SteSC}. 
The main result of that section applies to any hypersurface with isolated singularities. However, when applied to a hypersurface with an ordinary $m$-fold point we obtain:

\begin{theorem}[Corollary~\ref{corSpec}]
Let $X\subset \Ps^n$ be a hypersurface with an ordinary $m$-fold point and no other singularities. Suppose $X$ lies in the closure of the Severi variety of degree $d$ hypersurface with $\delta$ nodes. Then $d\geq m$ and 
\[ \delta_X \leq  \# \left\{ (k_1,k_2,\dots,k_n) \in (0,m)^n \cap \Z^n \mid \frac{1}{2}nm-m<\sum k_i \leq \frac{1}{2}nm \right\}. \]
In particular, if $n=3$ then 
\[ \delta_X\leq \left\{
\begin{array}{ll}
\frac{23}{48}m^3-m^2+\frac{7}{12} m &\mbox{ if } m\equiv 0\bmod 2\vspace{0.2cm}\\ 
\frac{23}{48}m^3-\frac{21}{16}m^2+\frac{49}{48}m-\frac{3}{16} &\mbox{ if } m\equiv 1 \bmod 2.\end{array}\right.\]
\end{theorem}

This bound is different from, but close to, Varchenko's upper bound on the number of nodes on a surface of degree $m$ in $\Ps^n$. This latter bound is a corollary of   Varchenko's  result  that  any open interval of length 1 is a semicontinuity set for the spectrum of  a low-weight deformation of a weighted homogeneous singularity. In our set-up the singularity is not always weighted homogeneous. To resolve this, we apply  Steenbrink's result that any half open interval of length 1 is a semicontinuity set for any deformation of an isolated singularity. This leads to an upper bound for $\delta_X$ which is slightly larger than the upper bound for the number of nodes on a degree $m$ surface in $\Ps^3$. However  both bounds have the same degree and the same leading coefficient, when considered as polynomials in $m$.

In the following two sections we prove three lower bounds for $\delta_X$. Denote with $\mathcal{T}_{d,m}^n \subset  \Ps(\C[x_0,\dots,x_n]_d)$ the locus of hypersurfaces in $\Ps^n$ of degree $d$ with an ordinary $m$-fold point as singularity.

In Section~\ref{secLower} we show the first lower bound:
\begin{theorem}[Theorem~\ref{thmLowerSur}]  Let $m,d$ be positive integers such that $ d\geq m>1$.
Let $\delta:=(\lfloor \frac{ m}{3}\rfloor+1)^{3}$ if $3\nmid d$ and  $\delta:=\frac{m}{3}( \frac{m}{3}+1)^2$  if $3\mid m$. 
 Then $\mathcal{T}^3_{d,m}$ is contained in the closure of $V^3_{d,\delta}$. 
 
 In other words, if $X$ be a surface of degree $d$ with an ordinary $m$-fold point and no further singularities, then there exists a family of degree $d$ surfaces such that the general fiber is a surface of degree $d$ with $\delta$ nodes and the special fiber is  $X$.
\end{theorem}

The construction uses the deformation result \cite[Theorem 3.3]{CiGa26}, which can only be applied to surfaces.

In Section~\ref{secLowSec} we prove results in arbitrary dimension. For this we use give an explicit deformations.

\begin{theorem}[Theorem~\ref{thmAll}] Fix integers $d,m,n$ with $n\geq 3$ and  $d\geq m>1$. Let $\delta:= \lfloor\frac{m+n}{n+1}\rfloor^n$.
 Then $\mathcal{T}^n_{d,m}$ is contained in the closure of $V^n_{d,\delta}$. 
Equicalently, if $X$ is a surface of degree $d$ with an ordinary $m$-fold point and no further singularities then $\delta_X\geq \delta$.
\end{theorem}

For $n=3$ (the surface case) this bound is slightly worse than the previous bound. However, the construction is more explicit and works in arbitrary dimension. Moreover, this construction produces examples with much higher values of $\delta$:

\begin{theorem}[Theorem~\ref{thmExist}] Fix integers $d,m,n$ with $n\geq 3$ and  $d\geq m>1$. Then there exists a hypersurface of degree $d$ in $\Ps^n$ with an ordinary $m$-fold point and no further singularities such that  $\delta_{X}\geq  \lfloor \frac{m}{2}\rfloor ^n$.
\end{theorem}

In Section~\ref{secCurve} we show that the spectral upper bound in the curve case yields precisely the well-known upper bound $\delta_C\leq p_a(C)-p_g(C)$.

\section{The spectrum of an isolated singularity and an upper bound for $\delta_X$}\label{secUpper}
Let $X$ be a projective hypersurface with isolated singularities. 
In this section we  explain how the spectra of the singularities of $X$ give an upper bound on $\delta_X$.
For a more extensive introduction to the spectrum of a singularity see \cite{DucoNotes} or \cite[Section II.8]{Kuli}. The former reference contains also a discussion of the spectral bound, and references to the original results.

The spectrum of a singularity can be both represented as a multisum of elements from $\Q$, or as ($\frac{1}{s}$ times) a polynomial in a fractional power of $s$. For expository reasons we prefer the latter choice, however, the above mentioned references prefer the former.

\begin{notation} Fix the following:
\begin{itemize}
\item $R= \Z[s^{\alpha}\colon \alpha \in \Q_{>0}]=\Z[s^{1/k}\colon k\in \Z_{>0}] $, the ring of spectral polynomials,
\item $f\in \C[x_1,\dots,x_n]$,  a polynomial with an isolated singularity at the origin,
\item $M(f)= \C\{x_1,\dots,x_n\}/(\frac{\partial f}{\partial x_1},\dots,\frac{\partial f}{\partial x_n})$, the Milnor algebra of $f$,
\item $\mu=\dim_{\C} M(f)$, the Milnor number of $f$,
\item $F$, the Milnor fiber of $f$,
\item $T^*$, the monodromy operator acting on $H^{n-1}(F)$.
\end{itemize}
For $\alpha \in \Q$, $-1<\alpha\leq 0$ denote with $C_{\alpha}$ the generalized eigenspace of $T^*$ acting on $H^{n-1}(F)$ for the eigenvalue $\exp(2\pi \sqrt{-1} \alpha)$. 
\end{notation}

\begin{remark}
 Consider the cohomology group  $H^{n-1}(F,\C)$. The dimension of this vector space equals the Milnor number $\mu$. 
The cohomology group $H^{n-1}(F,\C)$ has a natural mixed Hodge structure. Moreover, the monodromy operator $T$ induces an action $T^*$ on $H^{n-1}(F,\C)$, whose semisemplification has finite order. For $-1<\alpha\leq 0$ the generalized eigenspace  $C_{\alpha}$ has an induced  Hodge filtration.
\end{remark}
\begin{definition}
The \emph{spectrum} or \emph{spectral polynomial} of the isolated singularity $f$ is defined as 
\[ \sum_{-1<\alpha\leq 0}\sum_{p=0}^{n-1} \left(\dim_{\C} \Gr_F^{n-1-p} C_{\alpha}\right) s^{\alpha+p+1}\in R \]
and denoted by $\spe(f)$.
\end{definition}

\begin{remark}
The above definition yields the spectral polynomial associated with the Saito spectrum of the singularity. The Steenbrink spectrum can be obtained from the Saito spectum by shifting all spectral numbers  by one to the left. Hence the spectral polynomial associated with the Steenbrink spectrum  is the polynomial associated to the Saito spectrum divided by $s$, and is an element of $s^{-1}R$. 

Using the symmetry of Hodge numbers and the fact that the characteristic polynomial of the monomial operator $T$ is a polynomial with rational coefficients one obtains
\[ s^n\spe(f)\left(\frac{1}{s}\right)=\spe(f)(s),\]
i.e., the Saito spectrum is symmetric around $n/2$ and the Steenbrink spectrum is symmetric around $(n-2)/2$. For instance, in the curve case the Steenbrink spectrum is symmetric around 0.

By definition the (Saito) spectral polynomial is a polynomial in $s^{1/k}$ where $k$ is the order of the semi-simplification of $T$.
%
%One of the advantages of the spectral polynomial associated with the Saito spectrum rather than the Steenbrink spectrum happens for the Thom-Sebastiani Theorem: the Saito-spectral polynomial of the join of two singularities is the product of the spectral polynomials of the singularities..
\end{remark}

We will now focus on ordinary multiple points and (semi-)weighted homogeneous singularities.

\begin{definition} Let $f\in \C[x_1,\dots,x_n]$ be a polynomial such that $f(0)=0$.
A singularity $(f,0)$ is \emph{weighted homogeneous} if and only if it is right equivalent to a weighted homogeneous polynomial $(g,0)$.

A singularity $(f,0)$ is \emph{semi-weighted homogeneous} if the we can choose coordinates and weights such that lowest degree term in a graded decomposition of $f$ is a weighted homogeneous singularity.

An \emph{ordinary $m$-fold point} is an isolated singularity of multiplicity $m$ such that the tangent cone is a cone over a smooth hypersurface of degree $m$ in $\Ps^{n-1}$. 
\end{definition}

\begin{remark}
By a result of Saito  \cite{SaiMilTj} the singularity $(f,0)$ is weighted homogeneous if and only if $(f,0)$ is contact equivalent to a homogeneous polynomial if and only if the Milnor number of $f$ and the Tjurina number of $f$ coincide.

It is easy to construct examples of ordinary multiple points which have distinct Tjurina and Milnor numbers, hence which are not weighted homogeneous singularities.
\end{remark}
 
\begin{remark}
Let $f$ be a weighted homogeneous polynomial. The Milnor algebra of $f$ comes with a natural $\Q$-grading,  such that $f$ is homogeneous of degree 1.  For $\beta\in \Q$ we denote with $M(f)_{\beta}$ the graded part of degree $\beta$.
\end{remark}

\begin{proposition} Let $f$ be a weighted homogeneous polynomial in $x_1,\dots,x_n$ with rational weights $w_1,\dots,w_n$, chosen such that $\deg(f)=1$. Let $w=\sum_{i=1}^n w_i$. Then
\[ \spe(f)=\sum_{\alpha\in \Q_{>0}} \dim M(f)_{\alpha-w} s^{\alpha}.\]
\end{proposition}

\begin{proof}
We follow \cite[Section II.8.4]{Kuli}. One can identify the graded parts of the Milnor algebra with the graded parts $\Gr_F^p C_{\alpha}$. More precisely, for $p\in \Z$ and $-1<\alpha \leq 0$ one has that $\Gr_F^p C_{\alpha}=M(f)_{n-p+\alpha-w}$
The spectral polynomial now simplifies to
\[  \sum_{-1<\alpha\leq 0}\sum_{p=0}^{n-1} \left(\dim_{\C} M(f)_{p+1+\alpha-w}\right) s^{1+\alpha+p} \in R.\]
\end{proof}

\begin{remark} As remarked before not all ordinary $m$-fold point are weighted homogeneoust. 
If $(f,0)$ is an ordinary multiple point of order $m$ which is also weighted homogeneous, then  $\dim M(f)_{\beta}$ equals the coefficient of $t^{\beta m}$ in
\[ \frac{(1-t^m)^n}{(1-t)^n}.\]
At the same time, the Milnor algebra is the quotient of $\C[x_1,\dots,x_n]$ by a complete intersection ideal with all generators of  degree $m-1$ in $x_1,\dots,x_n$. The Koszul complex on these $n$ generators is a minimal resolution. Hence if $\beta m$ is an integer between $0$ and $n(m-1)$, then 
\[ \dim M(f)_{\beta}=\sum_{j=0}^{\lfloor \frac{\beta m}{m-1} \rfloor} (-1)^j \binom{n}{j} \binom{ \beta m -j(m-1)+(n-1)}{n-1}.\]
Lateron we explain that the spectrum of an ordinary multiple point of order $m$ is independent of whether the singularity is  weighted homogeneous or not. (Remark~\ref{RmkSemCon}.)
\end{remark}

\begin{example}
Let $f=\sum_{i=1}^n x_i^2$ be the defining polynomial of a node. Then the spectral polynomial equals $s^{n/2}$.

Let  $f=\sum_{i=1}^3 x_i^3$. Then $f$ is a 3-fold point on a surface, and is homogeneus. The spectral polynomial is $s+3s^{4/3}+3s^{5/3}+s^2$.  For $f=\sum_{i=1}^3 x_i^4$
we have  the spectral polynomial $s^{3/4}+3s+6 s^{5/4}+ 7 s^{3/2}+6 s^{7/4}+3s^2+s^{9/4}$.
\end{example}

We continue by  studying the behaviour of the spectrum in families.  

\begin{definition}
Let $p(s)\in R$ with $p(s) =\sum_{\alpha\in \Q} n_\alpha s^\alpha$. For a subset $I\subset \R$ we define the \emph{$I$-spectral degree} of $p$, as $\sum_{\alpha \in I \cap \Q} n_{\alpha}$ and  denote this by $\deg_I(p(s))$.

Let $\Delta\subset \C$ be an open disc containing $0$.
 Let $f_t, (t\in \Delta)$ be a family of  hypersurface singularities. Suppose that $\spe(f_t)$ is independent of $t$ for $t\neq 0$. We say that a set $S$ is a \emph{semicontinuity set} if $\deg_S \spe(f_0)\geq \deg_S \spe(f_t)$ for all $t\in \Delta$.
 If a semicontinuity  set $S$ is an interval then we use the notion \emph{semicontinuity interval}.
 \end{definition}

\begin{notation}
For $\alpha \in \R$ let  $I(\alpha)=\{x\in \R \mid \alpha <x\leq \alpha+1\}$ be the half-open interval of length 1, starting at $\alpha$.

If $f$ has an isolated hypersurface singularity then set   
\[ \delta(f):= \degsp_{I(\frac{n-2}{2})}(f).\]
\end{notation}
 
 \begin{remark}\label{RmkSemCon}
 Varchenko \cite{VarSC} proved that for every $\alpha\in \R$ the interval $(\alpha,\alpha+1)$ is a semicontinuity interval if the singularity is weighted homogeneous, Steenbrink \cite{SteSC} showed that for \emph{every} isolated singularity the half-open interval $I(\alpha)$ is a semicontinuity interval.
 
 From Steenbrink's result it follows that the spectrum is constant for  any $\mu$-constant deformation of a singularity. In particular, a semi-weighted homogeneous singularity has the same spectrum as the associated weighted homogeneous singularity.
 
 An ordinary $m$-fold point does not need to be semi-weighted homogeneous, but it always has a $\mu$-constant deformation to an ordinary $m$-fold point. In particular, any two ordinary $m$-fold points on an $n-1$-dimensional hypersurface have the same spectrum.
 \end{remark}

Steenbrink's result implies now the following upper bound for $\delta_X $.

\begin{proposition}\label{prpSpectral} Let $(X_t)_{t\in \Delta}\subset \Ps^n$ be a family of hypersurfaces of degree $d$ with isolated singularities. Suppose that $X_t$ is nodal for $t\neq 0$ and has precisely $\delta$ nodes  and $X_0$ has  isolated singularities at  $p_1,\dots p_k$ with local equations $f_{p_i}=0$ and no further singularities. Then
\[ \delta \leq \sum_{i=1}^k \delta(f_{p_i}).\]
\end{proposition}

\begin{proof} The spectrum of a node consists only of  the number $\frac{n}{2}$. By semicontinuity  \cite{SteSC} we have that 
\[  \sum_{i=1}^k \delta(f_{p_i})= \sum_{i=1}^k  \degsp_{I(\frac{n-2}{2})}(f_{p_i})\geq \delta.\]
\end{proof}

Let $f$ be an $m$-fold point, then $\degsp_{I((n-2)/2)}(f)$ equals
\[ \# \left\{ (k_1,k_2,\dots,k_n) \in (0,m)^n \cap \Z^n \mid \frac{1}{2}nm-m<\sum k_i \leq \frac{1}{2}mn \right\}. \]

The Arnol'd number
\[ A_n(m):=  \# \left\{ (k_1,k_2,\dots,k_n) \in (0,m)^n \cap \Z^n \mid \frac{1}{2}nm-m+1<\sum k_i \leq \frac{1}{2}mn \right\} \]
is smaller than $\degsp_{I((n-2)/2)}(f)$. Varchenko \cite{VarSC} showed that the Arnol'd number is an upper  bound for the number of nodes on a hypersurface of degree $d$. In Varchenko's case one uses the co-called Bruce deformation see \cite[Section 5.4]{DucoNotes}, which is of low weight. In our case we cannot assume that the deformation is such that we can apply Varchenko's result.

\begin{corollary}\label{corSpec}
Let $X\subset \Ps^n$ be a hypersurface with an ordinary $m$-fold point and no other singularities. Suppose $X$ lies in the closure of the Severi variety of degree $d$ hypersurface with $\delta$ nodes. Then $d\geq m$ and 
\[ \delta \leq  \# \left\{ (k_1,k_2,\dots,k_n) \in (0,m)^n \cap \Z^n \mid \frac{1}{2}nm-m<\sum k_i \leq \frac{1}{2}nm \right\}. \]
In particular, if $n=3$ then 
\[ \delta\leq \left\{
\begin{array}{ll}
\frac{23}{48}m^3-m^2+\frac{7}{12} m &\mbox{ if } m\equiv 0\bmod 2\vspace{0.2cm}\\
\frac{23}{48}m^3-\frac{21}{16}m^2+\frac{49}{48}m-\frac{3}{16} &\mbox{ if } m\equiv 1 \bmod 2.\end{array}\right.\]
\end{corollary}

\begin{proof}
Let $f=\sum_{i=1}^n x_i^m$. Then the Jacbian ideal of $f$ is generated by $x_i^{m-1}$, where $t=1,\dots n$. Therefore  $\C$-vector space $x_1x_2\dots x_n M(f)$ is spanned by monomials of the form $\prod_{i=1}^n x_i^{k_i}$ with $0<k_i<m$ for $I=1,\dots n$.
The spectral degree of $\prod x_i^{k_i}$ equals $\frac{1}{m}\sum k_i$. By the previous proposition $\delta$ is bounded by the number of such monomials where the spectral degree is between $\frac{1}{2}n-1+\frac{1}{n}$ and $\frac{1}{2}n$. A straightforward computation gives the values for $n=3$.
\end{proof}

\begin{remark} This proposition shows that if $S$ is a surface with ordinary multiple point and no further singularities then  $\delta_S\leq 4$ for a triple point, $\delta_S\leq 17$ for a quadruple point and $\delta_S\leq 32$ for a quintuple point. 
One the other hand, Ciliberto and Galati showed  \cite{CiGa26} that $\delta_S=4$, $\delta_S\geq 16, \delta_S\geq 31$ for triple, quadruple and quintuple points respectively.

Moreover, the Varchenko bounds for degree $4$ and $5$ surfaces in $\Ps^3$ are $16$ and $31$ respectively. Therefore, if  there exists an example of a surface with a 4-fold point and $\delta=17$ or a surface with a 5-fold point and  $\delta=32$ then the corresponding  deformation cannot be of low weight, or the singularity cannot be weighted homogeneous. For a similar reason we cannot apply \cite[Theorem 3.3]{CiGa26} to construct examples since this requires a degree $m$ surface with $\delta$ nodes, which is then bounded by the Varchenko bound.
\end{remark}

\begin{remark} We call an $m$-fold point on a surface a \emph{quasi-ordinary} $m$-fold point  if the exceptional divisor of the blow-up of the singularity is a nodal curve, and if one blow-up suffices to resolve the singularity. 
Consider the following polynomials
\[ x^3+y^3+z^3+f_4, (x^3+x^2z+y^2)+f_4, z(x^2+y^2+z^2)+f_4, xyz+f_4\]
where $f_4$ is a generic homogeneous polynomial of degree 4. Then these are ordinary or quasi ordinary 3-fold points.

These singularities have Milnor numbers 8,9,10,11 respectively. One easily checks that the value of $\delta(f)$ for there singularities is $4,5,6,7$ respectively.
Ciliberto--Galati \cite{CiGa26} constructed for every $d\geq 4$ a family of 7-nodal hypersurfaces degenerating in a surface with a single triple point, with tangent cone $xyz=0$ and such that one blow-up suffices to resolve this singularity. 
Since $\delta(f)=7$ for such singularities  it is not possible to have a 8-nodal surface degenerating in such a singularity, hence their construction is optimal.
\end{remark}

\section{Lower bound for $\delta_S$}\label{secLower}
Let $S\subset \Ps^3$ be a surface of degree $d$ with a unique singular point, which is an ordinary $m$-fold point.

In \cite{CiGa26} it is shown that $\delta_S\geq \binom{m-1}{2}=\frac{m^2-m}{2}$. In the previous section we presented an upper bound which is cubic in $m$. 
In this section we show that  $\delta_S\geq \frac{m^3}{27}$.
%If one could extend the result \cite[Theorem 3.3]{CiGa26} to higher dimension, then this would yield a lower $\delta_X\geq  \frac{m^n}{n^n}$, for a hypersurface $X\subset \Ps^n$ of degree $d\geq m$ containing an ordinary $m$-fold point and no further singularities.

We start by a recalling a result from commutative algebra:

\begin{proposition} Let $\Sigma\subset \Ps^n$ be a zero-dimensional scheme, let $I=I(\Sigma)\subset \C[x_0,\dots,x_n]$ be the ideal of $\Sigma$ and let 
\[ 0 \to \oplus_{j=1}^{i_t} S(-a_{tj}) \to \dots \to \oplus_{j=0}^{i_1}S(-a_{1j}) \to S \to S/I \to 0\]
be a free graded resolution of $S/I$, where $S=\C[x_0,\dots, x_n]$ and the $a_{ij}$ are positive integers.
Then $H^1(\Ps^n,\cO(d)\otimes \cI_{\Sigma})=0$ for all $d\geq \max(a_{ij})-n$.
\end{proposition}

\begin{proof} Since $H^i(\Ps^n,\cO(d))=0$ for all $i,d>0$ and $H^0(\Ps^n,\cO\times \cI_{\Sigma})=I_d$ we obtain that
\[ 0 \to H^0(\cO(d)\otimes \cI_{\Sigma}) \to H^0(\cO(d))) \to H^0(\cO_{\Sigma}) \to H^1(\cO(d)\otimes \cI_{\Sigma}) \to 0 \] is exact.  
Hence  $H^1(\Ps^n,\cO(d)\otimes \cI_{\Sigma})=0$ if and only if $h_I(d)=\length \Sigma$. i.e., the Hilbert function and Hilbert polynomial of $I$ agree at $d$.

One has that $\dim S_k=\frac{(k+n)(k+n-1)\dots (k+1)}{n!}$ for $j\geq -n$. Hence the Hilbert polynomial and Hilbert function of $\Ps^n$ agree for $k\geq -n$. From this it follows directly that the Hilbert polynomial and Hilbert function of $S/I$ agree for all $d$ at least $\max(a_{ij})-n$.
\end{proof}

\begin{proposition}\label{prpHilbBound} Let $\Sigma\subset \Ps^n$ be a zero-dimensional scheme, which is a scheme-theoretic complete  intersection of multidegree $(d_1,\dots,d_n)$. 
Then $H^1(\Ps^n,\cO(d)\otimes \cI_{\Sigma})=0$ for all $d\geq (\sum_{i=1}^n d_i)-n$.
\end{proposition}
\begin{proof}
Let $F_1,\dots,F_n$ be  homogeneous generators of $I$. Then the Koszul complex on $F_1,\dots,F_n$ is a resolution of $S/I$. In particular, for this resolution one has that $a_{ij}\leq \sum_{k=1}^n d_k$ for all $i$ and $j$. The result now follows from the previous result.
\end{proof}

\begin{theorem}\label{thmLowerSur} Let $n,d\geq 3$ be integers.  Moreover, suppose $(n,d)\neq (3,3)$. 

There exists a nodal hypersurface $X\subset\Ps^{n}$ of degree $d$ with precisely $(\lfloor \frac{d}{n} \rfloor+1)^{n}$ nodes such that $H^1(I_{\Sigma} (d))=0$, where $\Sigma=X_{\sing}$. If, moreover, $n$ does not divide $d$, then  $H^1(I_{\Sigma} (d-1))=0$.

If $n$  divides $d$ then there exists a nodal hypersurface $X\subset\Ps^{n}$ of degree $d$ with precisely $\frac{d}{n} (\frac{d}{n}+1)^{n-1}$ nodes such that $H^1(I_{\Sigma} (d-1))=H^1(I_{\Sigma} (d))=0$. 
\end{theorem}

\begin{proof}
Fix $n$ general forms $F_1,\dots,F_n$ in $x_0,\dots,x_n$ of degree $\lfloor \frac{d}{n} \rfloor+1$. Let $I=(F_1,\dots,F_{n})$. Then $I$ is a complete intersection ideal and its vanishing locus in $\Ps^n$ is finite.
Consider the ideal  $I^2=(F_iF_j)_{1\leq i\leq j\leq n}$. Pick a general 
\[F=\sum_{1\leq i\leq j \leq n+1} G_{ij}F_iF_j \in (I^2)_d\] and let $X=V(F)$.
The generators of $I^2$ have  degree  at most $2\lfloor d/n \rfloor+2$. This is at most  $d$ for our choices of $n$ and $d$.
Hence the base locus of $(I^2)_d$ is precisely $V(I)$ and by Bertini's theorem $X$ is smooth outside $V(I)$, hence $\Sigma=X_{\sing}=V(I)$. It remains to check that each singular point is a  node. For each point in $\Sigma$, the dehomogenization of the  polynomials $F_1,\dots F_{n}$ yields a  system of local coordinates. Recall that matrix  $\frac{ \partial^2 F}{\partial F_i\partial F_j}$ equals the matrix with entries $G_{ij}$ for $i<j$, $G_{ji}$ for $i>j$ and $2G_{ii}$ for $i=j$.
In order to have an node at every point in $\Sigma$ we need that the determinant of this matrix is invertible modulo $I$, which holds true for a generic choice of the $G_{ij}$.

From Proposition~\ref{prpHilbBound} it follows that $H^1(I_{\Sigma}(k))=0$ for $k\geq n \lfloor \frac{d}{n} \rfloor $. Since $n \lfloor \frac{d}{n} \rfloor\leq d$. This proves the claim on $H^1(I_\Sigma(d))$.
 
 If $H^1(I_{\Sigma}(d-1))\neq 0$ then equality holds in the previous formula, in particular, $n \lfloor \frac{d}{n}\rfloor =d$, which is in turn equivalent to  $n\mid d$. Hence $H^1(I_{\Sigma}(d-1))=0$ if $n\nmid d$. In the case $d\mid n$  we replace $F_{n}$ by a form of one degree less and obtain also  $H^1(I_{\Sigma}(d-1))=0$. 
\end{proof}

In \cite{CiGa26} a component of the Veronese variety $V^3_{d,\delta}$ is called regular if for a general $[X]\in V^3_{d,\delta}$ with $\Sigma$ the singular locus of $X$ we have $h^1(\Ps^3,\cO(d)\otimes \cI_{\Sigma})=0$. If is called very regular if $h^1(\Ps^3,\cO(d-1)\otimes \cI_{\Sigma})=0$. In \cite[Theorem 3.3]{CiGa26}
they show that if $\delta$ is chosen such that $V^3_{m,\delta}$ has a very regular component then the locus of degree $d$ surfaces with an $m$-fold point is contained in the closure of $V^3_{d,\delta}$ for all $d\geq m$ 

\begin{corollary} Let $X$ be a surface of degree $d$ with an ordinary $m$-fold point and no further singularities. Then there exists a family of degree $d$ surfaces with $\delta:=(\lfloor \frac{ m}{3}\rfloor+1)^{3}$ nodes degenerating to $X$ if $3\nmid m$ or $\delta:=\frac{m}{3}( \frac{m}{3}+1)^2$ nodes if $3\mid m$. 
\end{corollary}

\begin{proof}
From the previous proposition it follows that the Severi variety $V_{m,\delta}^3$ has a very regular component.  The result follows from \cite[Theorem 3.3]{CiGa26}.
\end{proof}

\begin{remark} Combining the above results we obtain polynomial upper and lower bounds for $\delta_S$, where the lower bound is asymptotically $\frac{1}{27}m^3$ and the upper bond is asymptotically $\frac{23}{48} m^3$.

One might attempt to use other constructions of surfaces of degree $m$ with many nodes in order to produce surfaces $S$ of degree $d$ with an ordinary $m$-fold point and large $\delta_S$.
The best known asymptotic constructions of degree $m$ surfaces with nodes seem to have  $\frac{5}{12} m^3$ nodes.
However, in this construction the ideal of the nodes  has syzygies in degree $\lfloor \frac{3m}{2}\rfloor$, and therefore the nodes do not impose independent conditions in all degrees up to  $\lfloor \frac{3m}{2}\rfloor-4$. In particular, $S$ does not lie on a very regular component for $m>7$ and we cannot apply the result by Ciliberto--Galati.
\end{remark}

\section{Lower bound for $\delta_X$}\label{secLowSec}
In this section we study a certain family  of nodal hypersurfaces in $\Ps^n$  with as special fiber a hypersurface with an ordinary $m$-fold point.

We use this construction to show that  for any $X\subset \Ps^n$ of degree $d$ with an ordinary $m$-fold point and no further singularities we have  $\delta_X\geq \frac{(m+n)^n}{(n+1)^n}$.

Moreover, for fixed $n,m,d$ with $n\geq 3$ and $d\geq m$ we construct a hypersurface $X\subset \Ps^n$ of degree $d$ with an ordinary $m$-fold point and no further singularities such that $\delta_X\geq \left( \frac{m-1}{2} \right)^n$ if $m$ is odd and $\delta_X \geq \left( \frac{m}{2}\right)^n$ if $m$ is even.

We  use the following construction of a family of nodal hypersurfaces degenerating to a hypersurface with an $m$-fold point.

The idea is to take a nodal hypersurface $X_1$ whose singular locus is a complete intersection in $\Ps^n$ of length at most $(m/2)^n$  and contained in $\Ps^n\setminus V(x_0)$.  We then construct a family of hypersurfaces $X_t$ such that for all but finitely many $t$, we have  that the hypersurface $X_1$ has a node at $(1:a_1:\dots,a_n)$ if and only if $X_t$ has a node at $(1:ta_1:\dots:ta_n)$. Then at $t=0$ all the nodes collide in one singular point. We will then show that for a generic choice the singularity is an ordinary $m$-fold point. 

\begin{construction}\label{ConstMain}
Fix positive integers $n, d,d_1,\dots,d_n,m$ such that $m\leq d$ and for all $i$ we have $d_i\leq \frac{1}{2}m$. 

For $i=1,\dots,n$ let $F_i$ be a general form in $x_0,\dots,x_n$ of degree $d_i$, such that $V(F_1,\dots,F_n)$ is a  reduced complete intersection of length $\prod d_1,\dots,d_n$.
Choose general forms $G_{ij}$ in $x_0,\dots,x_m$ of degree $d-d_i-d_j\geq 0$. 

Set $F=\sum_{1\leq i \leq j \leq n} G_{ij}F_iF_j$.
For $t\in \C$ and $i=1,\dots,n$ define $\tilde{F_i}(x_0,\dots,x_n)$ as $F_i(tx_0,x_1,\dots, x_n)$. Write 
\[ G_{ij}=\sum_{k=0}^{d-d_i-d_j} x_0^{k}  g_{ijk}(x_1,\dots,x_n).\]
Then $g_{ijk}$ is a form of degree $d-d_i-d_j-k$.
Define \[\tilde{G}_{ij}=\sum_{k=0}^{d-d_i-d_j} x_0^k  t^{\max(0,k-d+m)} g_{ijk}(x_1,\dots,x_n).\]

Let $X_t=V(\sum_{1\leq i\leq j \leq n} \tilde{F_i}\tilde{F_j}\tilde{G}_{ij})$, and let $X_1=V(\sum_{1\leq i\leq j \leq n} F_iF_jG_{ij})=V(F)$.
\end{construction}

\begin{proposition} For a general choice of the $F_i$ and the $G_{ij}$ we have that for all but finitely many non zero values of $t$ the hypersurface  $X_t$ is a nodal hypersurface of degree $d$ with $\prod_{i=1}^n d_i$ nodes.
\end{proposition}
\begin{proof}
By Bertini's theorem one has that for fixed $F_1,\dots F_n$ and general $G_{ij}$ the surface $X_1$ is  singular at $\Sigma:=V(F_1,\dots,F_n)$ and smooth elsewhere. Moreover, if the $F_1,\dots,F_n$ form a set-theoretic intersection of length $\prod d_i$ then for general $G_{ij}$ we have that all singularities are nodes:
The dehomogenization of the  polynomials $F_1,\dots F_{n}$ form a local system of coordinates at each of the points of $\Sigma$.
Hence if we pick some $F=\sum_{1\leq i\leq j \leq n+1} G_{ij}F_iF_j$ then we have that the matrix  $\frac{ \partial^2 F}{\partial F_i\partial F_j}$ equals the matrix with entries $G_{ij}$ for $i<j$, $G_{ji}$ for $i>j$ and $2G_{ii}$ for $i=j$.
In order to have nodes everywhere we need that the determinant of this matrix is invertible modulo $I_{\Sigma}$, which holds true for a generic choice of the $G_{ij}$.

For $t\neq 0$ we have that $X_t$ is singular at $\Sigma_t=V(\tilde{F}_1,\dots,\tilde{F}_n)$, and for all but finitely many values of $t$ we have that the singular locus is precisely this locus.

 For $t\neq 0$ the set $\Sigma_t$ is projectively equivalent to $\Sigma_1$, hence for all but finitely many $t$ we have that $X_t$ has $\prod_{i=1}^n d_i$  nodes and no further singularities.
\end{proof}

\begin{proposition} For a general choice of the $F_i$ and the $G_{ij}$, the hypersurface $X_0$ has an ordinary $m$-fold point and no further singularities.
\end{proposition}

\begin{proof} Using the generality assumption we have that the hypersurface $X_0$ has a singularity at $\Sigma:=V(F_1(0,x_1,\dots,x_n),\dots,{F}_n(0,x_1,\dots,x_n))$ and is smooth elsewhere.  The scheme  $\Sigma$ is supported at $(1:0:\dots:0)$.

Recall that  $X_0$ is the zero set of 
\[ \sum_{1\leq i \leq j \leq n} F_i(0,x_1,\dots,x_n)F_j(0,x_1,\dots,x_n)\left(\sum_{k=0}^{d-m} x_0^{k} g_{ijk}(x_1,\dots,x_n)\right)\]
This hypersurface has an $m$-fold point at $(1:0:\dots:0)$. The tangent cone $C_pX_0$ is given by
\[ \sum_{1\leq i\leq j\leq n} F_i(0,x_1,\dots,x_n)F_j(0,x_1,\dots,x_n) g_{ij(d-m)}(x_1,\dots,x_n)\]
Using the generality assumption we have that  the $f_i(0,x_1,\dots,x_n)$ do not have a common base point in $\Ps^{n-1}$. From Bertini's theorem it follows  that  a general choice of $g_{ijk}$  that the tangent cone $C_pX_0$ is a cone over a smooth hypersurface of degree $m$ in $\Ps^{n-1}$, in particular $X_0$ has an ordinary $m$-fold point at $(1:0:\dots:0)$. 
Applying  Bertini's theorem again, we obtain that for a general choice of the $G_{ij}$ the hypersurface $X_0$ has no singularities outside $\Sigma$.
\end{proof}

\begin{theorem}\label{thmExist} Fix integers $d,m,n$ with $n\geq 3$ and  $d\geq m>1$. Then there exists a hypersurface of degree $d$ in $\Ps^n$ with an ordinary $m$-fold point such that  $\delta_{X}\geq \left(  \lfloor \frac{m}{2}\rfloor \right)^n$.
\end{theorem}

\begin{proof} Pick $F_1,\dots,F_n$ generic forms of degree $\lfloor \frac{m}{2} \rfloor$ and $G_{ij}$ sufficiently generic forms of degree $d-2\lfloor \frac{m}{2}\rfloor$. Then $X_t$ is a family of singular hypersurfaces, with the property that the generic $X_t$ has $(\lfloor \frac{m}{2} \rfloor)^n$ nodes and no further singularities, and $X_0$  has an ordinary $m$-fold point and no further singularities.
\end{proof}

\begin{remark}
In the surface case we obtain examples  with  $\delta_X\geq \frac{1}{8} m^3$  for $m$ even and $\delta\geq \frac{1}{8}(m-1)^3$ for $m$ odd. However, the singular locus of the deformation  is a complete intersection of three surfaces of degree $\frac{m}{2}$ or $\frac{m-1}{2}$. This means that the highest syzygy is in degree $\frac{3m}{2}$, respectively $\frac{3m-3}{2}$. In particular, for $m\geq 10$ we have that $H^1(I_{\Sigma}(m))\neq 0$, hence we cannot apply  \cite[Theorem 3.3]{CiGa26} directly.

If $n>3$ then we have that $2^n\leq n!$. In particular, for fixed $n$ and $m$ sufficiently large we find that $\#\Sigma \geq \binom{n+m}{n}=h^0(\mathcal{O}_{\Ps^n}(m))$. Therefore $h^1(I_{\Sigma}(m))\neq 0$ for cardinality reasons. This implies that $V^n_{m,\delta}$ does not contain any very regular component.% In  order to find better lower bounds for $\delta$  by using a  generalisation of  \cite[Theorem 3.3]{CiGa26} one would need to weaken the general position requirement.
\end{remark}

The previous results produced for fixed $(n,d,m)$ an example $X$ for which the lower bound for $\delta_X$ holds true.
We  proceed by producing a lower bound for $\delta_X$ where $X$ is  a general hypersurface of degree $d$ with an $m$-fold point.

\begin{notation}
For a homogeneous ideal $I\subset \C[x_0,\dots,x_n]$, denote with $h_I(d)$ the Hilbert function of $I$ at $d$, namely $\dim_{\C} (S/I)_d$.
\end{notation}

We start with an auxiliary lemma:

\begin{lemma} Let $F_1,\dots,F_n\in \C[x_1,\dots,x_n]$ be a regular sequence of homogeneous forms of degree $d_1,\dots,d_n$, such that $1\leq d_1\leq d_2\leq \dots\leq d_n$. Let $I=(F_1,\dots,F_n)$.

Then  $h_{I^{k}}(d)=0$ for all $k\geq 1$ and $d\geq (k-1)d_n-n+ \sum_{i=1}^n d_i$.
\end{lemma}

\begin{proof} We prove the claim by induction on $k$.
Suppose first that $k=1$. Recall that $F_1=F_2=\dots=F_n=0$ defines a scheme-theoretic complete intersection in $\As^n$, hence is finite, and since the $F_i$ are homogeneous this scheme is concentrated at the origin.
Hence $h_I(d)=0$ for $d\geq \sum_{i=1}^n d_i-n$,

Next, suppose  that $k\geq 2$ holds. Let $G$ be a homogeneous form of degree $d$ at least $(k-1)d_n-n+\sum (d_i)$. Then $h_I(d)=0$. Hence there exists forms $G_1,\dots, G_n$  of degree at least $(k-1)d_n-d_i-n+\sum d_i\geq (k-2)d_n-n+\sum d_i$ such that
\[ G=\sum_{i=1}^n G_iF_i.\]
By induction $G_i\in  I^{k-1}$ for $i=1,\dots,n$ and  therefore $G\in I^k$. Hence $h_{I^k}(d)=0$ for $d\geq (k-1)d_n+ \sum_{i=1}^n d_i-n$.
\end{proof}

\begin{theorem}\label{thmAll} Let $X\subset \Ps^n$ be a generic hypersurface of degree $d$ with an ordinary $m$-fold point and no further singularities. Then $\delta_X \geq \frac{(m+n)^n}{(n+1)^n}$
\end{theorem}

\begin{proof}
After a change of coordinates we may assume that the singular point of $X$ is at $(1:0:0:\dots:0)$. Then $X$ is the zero set of
\[ \sum_{j=0}^{d-m} x_0^j h_j(x_1,\dots,x_n),\]
where $h_j$ is homogeneous of degree $d-j$.

We want  to show that $X$  can be obtained as the central fiber in Construction~\ref{ConstMain}.

Pick  general forms $F_1,\dots,F_n$ of degree $\lfloor \frac{m+n}{n+1}\rfloor$. Let $I=(F_1,\dots,F_n)$. We aim to choose 
 the $g_{ijk}$ such that the resulting $X_0$ equals $X$.
Recall that $X_0$ is given by
\[ \sum_{1\leq i \leq j \leq n} F_i(0,x_1,\dots,x_n)F_j(0,x_1,\dots,x_n)\left(\sum_{k=0}^{d-m} x_0^{k} g_{ijk}(x_1,\dots,x_n)\right).\]
Hence $X$ is the central fiber from  Construction~\ref{ConstMain} if we can find for all $k$ between $0$ and $d-m$ forms $g_{ijk}$ such that
\[ \sum_{1\leq i \leq j \leq n} F_i(0,x_1,\dots,x_n)F_j(0,x_1,\dots,x_n) g_{ijk}(x_1,\dots,x_n) =h_k.\]
The right hand side is a homogeneous form of degree $d-k\geq m$. From the previous lemma it follows that $h_{I^2}(d')=0$ for $d'\geq (n+1) \lfloor \frac{m+n}{n+1}\rfloor -n$. In particular $h_{I^2}(d-k)=0$ and therefore such the $g_{ijk}$ exist.
In particular,  $\delta_X \geq \lfloor \frac{(m+n)}{(n+1)}\rfloor ^n$
\end{proof}

\begin{remark}  A generalization of \cite[Theorem 3.3]{CiGa26} to arbitrary dimension would likely contain the hypothesis $h^1(\Ps^n,\cO(d-1)\otimes \cI_{\Sigma})=0$. This hypothesis forces the number of nodes to be at most 
\[ \binom{m+n}{m}\]
which for fixed $n$  is asymptotically $\lfloor \frac{m}{n}\rfloor^n$. This bound is asymptotically   better than our bound $\lfloor \frac{m+n}{n+1}\rfloor^n$.
\end{remark}

\section{Curve case}\label{secCurve}

In this section we show that the spectral bound from Proposition~\ref{prpSpectral} yields the well-known upper bound for $\delta$ in the curve case.

\begin{lemma} Let $f\in \C[x_1,x_2]$ be an isolated singularity in the origin. Write $\spe(f)=\sum_{\alpha} n_{\alpha} t^\alpha$. Then $\delta(f)= \frac{\mu+n_1}{2}$.
\end{lemma}

\begin{proof}
As explained in Section~\ref{secUpper}
the total degree of the spectrum equals the Milnor number of $f$. Moreover the spectrum is symmetric around $1$ and lies between $0$ and $2$. Hence
\[ \degsp_{I(0)}(f_p)=\frac{\mu}{2}+\frac{n_1}{2}.\] 
\end{proof}

\begin{lemma} Suppose $C$ is an irreducible singular plane curve. For $p\in C_{\sing}$ let $f_p$ be a local equation for $C$ at $p$. Then
\[ p_a(C)-p_g(C)\geq \sum_{p\in C_{\sing}} \delta(f_p).\]
\end{lemma}

\begin{proof} Let $C'$ be a smooth curve of the same degree as $C$. 
Recall that $h^0(C)=h^2(C)=1$ by irreducibility and that $e(C)=e(C')+\mu(C)$ \cite[Corollary 5.4.4]{Dim}. This implies that $h^1(C)=h^1(C')-\mu$. Let $\tilde{C}$ be a resolution of singularities of $C$. Then each singular point $p$ of $C$ is replaced by $n_1(f_p)$ points.
Hence  $h^1(\tilde{C})=h^1(C)-n_1=h^1(C')-\mu-n_1$. In particular we have $2p_g(C)=2p_a(C)-\mu-n_1\leq 2p_a(C)-2\sum \delta(f_p)$.
\end{proof}

The following proposition shows that the spectral upper bound in the curve case gives precisely the well-known upper $\delta_C\leq p_a(C)-p_g(C)$:

\begin{proposition} We have  $\delta_C\leq \sum_{p\in C_{\sing}} \delta(f_p)\leq p_a(C)-p_g(C)$.
\end{proposition}
\begin{proof}
Combine Proposition~\ref{prpSpectral} with the previous Lemma.
\end{proof} 
\begin{remark} As explained in \cite{CiGa26}, in the curve case we have $\delta_C=p_a(C)-p_g(C)$. Hence the spectral bound is sharp in the curve case. It is widely expected that in higher dimension the spectral bound is not sharp for most choices of $(d,n)$.
\end{remark}
\bibliographystyle{plain}
\bibliography{remke2}

\begin{thebibliography}{1}

\bibitem{CiGa26}
C.~Ciliberto and C.~Galati.
\newblock On nodal deformations of singular surfaces in $\mathbb{P}^3$.
\newblock Preprint, available at \texttt{https://arxiv.org/abs/2602.09177},
  2026.

\bibitem{Dim}
A.~Dimca.
\newblock {\em Singularities and topology of hypersurfaces}.
\newblock Universitext. Springer-Verlag, New York, 1992.

\bibitem{Kuli}
V.~S. Kulikov.
\newblock {\em Mixed {H}odge structures and singularities}, volume 132 of {\em
  Cambridge Tracts in Mathematics}.
\newblock Cambridge University Press, Cambridge, 1998.

\bibitem{SaiMilTj}
K.~Saito.
\newblock Quasihomogene isolierte {S}ingularit\"aten von {H}yperfl\"achen.
\newblock {\em Invent. Math.}, 14:123--142, 1971.

\bibitem{SteSC}
J.~H.~M. Steenbrink.
\newblock Semicontinuity of the singularity spectrum.
\newblock {\em Invent. Math.}, 79:557--565, 1985.

\bibitem{DucoNotes}
D.~van Straten.
\newblock The spectrum of hypersurface singularities.
\newblock Preprint, available at \texttt{https://arxiv.org/abs/2003.00519},
  2020.

\bibitem{VarSC}
A.~N. Varchenko.
\newblock Semicontinuity of the spectrum and an upper bound for the number of
  singular points of the projective hypersurface.
\newblock {\em Dokl. Akad. Nauk SSSR}, 270:1294--1297, 1983.

\end{thebibliography}

\end{document}